\documentclass[a4paper]{amsart} 

\newcommand{\SCI}{\operatorname{SCI}}
\newcommand{\diam}{\operatorname{diam}}
\renewcommand{\i}{\mathrm{i}}

\DeclareMathOperator{\Res}{Res}

\newcommand{\f}{\frac}

\newcommand{\im}{\operatorname{Im}}

\newcommand{\R}{\mathbb R}
\newcommand{\C}{\mathbb C}
\newcommand{\N}{\mathbb N}
\newcommand{\Z}{\mathbb Z}

\newcommand{\re}{\operatorname{Re}}
\newcommand{\eps}{\varepsilon}
\renewcommand{\epsilon}{\varepsilon}

\newcommand{\Om}{\Omega}

\newcommand{\supp}{\operatorname{supp}}
\newcommand{\dist}{\operatorname{dist}}

\newcommand{\h}{\mathcal{H}}

\newenvironment{enumi}{\begin{enumerate}[(i)]}{\end{enumerate}}

\def\Xint#1{\mathchoice
{\XXint\displaystyle\textstyle{#1}}%
{\XXint\textstyle\scriptstyle{#1}}%
{\XXint\scriptstyle\scriptscriptstyle{#1}}%
{\XXint\scriptscriptstyle\scriptscriptstyle{#1}}%
\!\int}
\def\XXint#1#2#3{{\setbox0=\hbox{$#1{#2#3}{\int}$ }
\vcenter{\hbox{$#2#3$ }}\kern-.6\wd0}}

\def\dashint{\Xint-}

\AtBeginDocument{

}

\usepackage{etex}
\usepackage[english]{babel}
\usepackage{tikz}
\usetikzlibrary{arrows,decorations.pathmorphing,backgrounds,positioning,fit,petri,patterns}
\usepackage{pgfplots}
\usepackage[utf8]{inputenc}
\usepackage{graphicx}
\usepackage[colorlinks=true,linkcolor=red,citecolor=blue]{hyperref}
\usepackage{amsfonts}
\usepackage{enumerate}
\usepackage{booktabs}
\usepackage{array} 
\usepackage{paralist} 
\usepackage{subcaption} 
\usepackage{mathtools}
\usepackage{tabu}
\usepackage{amsthm}
\usepackage{empheq}
\usepackage{amsopn}
\usepackage{dsfont}
\usepackage{wrapfig}
\usepackage[font=small]{caption}
\usepackage{units}
\usepackage[symbol]{footmisc}
\usepackage{todonotes}
\usepackage{amssymb}
\usepackage{pdfpages}
\usepackage{setspace}
\usepackage{scrextend}
\usepackage[outline]{contour}
\usepackage{soul}
\usepackage{sidecap}
\usepackage{kvsetkeys}
\usepackage{etexcmds}

\makeatletter
\newcount\arg@count
\newcommand{\arg@parser}[1]{%
  \advance\arg@count\@ne
  \expandafter\let\csname arg\romannumeral\arg@count\endcsname\comma@entry
}
\newcommand\res[1]{
  \arg@count=\z@
  \comma@parse{ \lambda,A }\arg@parser 
  \arg@count=\z@
  \comma@parse{#1}\arg@parser
  \ifnum\arg@count>2 %
    \@latex@error{Too many arguments}{%
      The macro \string\mycmd\space got \the\arg@count\space
       arguments,\MessageBreak
      but expected are 2 arguments.\MessageBreak
      \@ehd
    }%
  \fi
  \edef\process@me{%
    \noexpand\@res
    {\etex@unexpanded\expandafter{\argi}}%
    {\etex@unexpanded\expandafter{\argii}}%
  }%
  \process@me
}
\newcommand{\@res}[2]{%
  \ensuremath\left( #1 - #2 \right)^{-1}
}
\makeatother

\makeatletter
\newcount\arg@count
\newcommand\p[1]{
  \arg@count=\z@
  \comma@parse{  }\arg@parser 
  \arg@count=\z@
  \comma@parse{#1}\arg@parser
  \ifnum\arg@count=2 %
  \else
    \@latex@error{Wrong number of mandatory arguments}{%
      The macro \string\p\space got \the\numexpr\arg@count-2\relax\space
      mandatory arguments,\MessageBreak
      but expected are 3 mandatory arguments.\MessageBreak
      \@ehd
    }%
  \fi
  \edef\process@me{%
    \noexpand\@p
    {\etex@unexpanded\expandafter{\argi}}%
    {\etex@unexpanded\expandafter{\argii}}%
  }%
  \process@me
}
\newcommand{\@p}[2]{%
  \ensuremath \left\langle #1 , #2 \right\rangle
}
\makeatother

\allowdisplaybreaks
\numberwithin{equation}{section}

\theoremstyle{definition}
\newtheorem{de}{Definition}[section]

\theoremstyle{plain}

\newtheorem{lemma}[de]{Lemma}
\newtheorem{theorem}[de]{Theorem}

\newtheorem{cl}[de]{Claim}

\theoremstyle{remark}
\newtheorem{remark}[de]{Remark}

\title{Computing scattering resonances}
\author{Jonathan Ben-Artzi}
\email{Ben-ArtziJ@cardiff.ac.uk}
\author{Marco Marletta}
\email{MarlettaM@cardiff.ac.uk}
\author{Frank R\"{o}sler}
\email{RoslerF@cardiff.ac.uk}
\thanks{JBA and FR acknowledge support from an Engineering and Physical Sciences Research Council Fellowship (EP/N020154/1).}
\address{School of Mathematics, Cardiff University, Senghennydd Road, Cardiff CF24 4AG, Wales, UK}
\date\today
\keywords{Scattering resonance, Solvability Complexity Index, Computational complexity}
\subjclass[2010]{47N40, 68Q25,  35B34}
\begin{document}

\maketitle

\begin{abstract}
The question of whether it is possible to compute scattering resonances of Schr\"odinger operators -- independently of the particular  potential -- is addressed. A positive answer is given, and it is shown that the only information required to be known \emph{a priori} is the size of the support of the potential. The potential itself is merely required to be $\mathcal{C}^1$. The proof is constructive, providing a universal algorithm which only needs to access the values of the potential at any requested point.
\end{abstract}

\setcounter{tocdepth}{1}

\section{Introduction and Main Result}
This paper provides an affirmative answer to the following question:
	\begin{quote}\emph{Does there exist a  universal algorithm for computing the  resonances of Schr\"odinger operators with complex potentials?}
	\end{quote}

To the authors' best knowledge this is the first time this question is addressed. Furthermore, the proof of existence  provides an actual algorithm (that is, the proof is constructive). We test this algorithm on some standard examples, and compare to known results.

The framework required for this analysis is furnished by the \emph{Solvability Complexity Index} ($\SCI$), which is an abstract theory for the classification of the computational complexity of problems that are infinite-dimensional. This framework has been developed over the last decade by Hansen and collaborators (cf. \cite{Hansen11,AHS,Ben-Artzi2015a}) and draws inspiration from   the seminal result \cite{DM} on solving quintic equations via a \emph{tower of algorithms}. {\bf We therefore emphasize that ours is an abstract result in analysis,  not  in numerical analysis.}

\subsection{Quantum Resonances}\label{sec:background}
Let us first define what a quantum resonance is. Let $q:\R^d\to\C$ be compactly supported, let
	\begin{equation*}
	H_q:=-\Delta+q
	\end{equation*}
be the associated Schr\"odinger operator in $L^2(\R^d)$ and let $\chi:\R^d\to\R$ be some compactly supported function with $\chi\equiv 1$ on $\supp(q)$. 
It follows from the explicit form of the free fundamental solution (cf. eq. \eqref{eq:Fundamental_Solution} below) that the map
	\begin{equation*}
	z\mapsto I + q(-\Delta-z^2)^{-1}\chi
	\end{equation*}
is an analytic operator-valued function on $\C\setminus\{0\}$. We define:

\begin{de}[Resonance]\label{def:res}
 A \emph{resonance} of $H_q$ is defined
to be a pole of the meromorphic operator-valued function $z\mapsto (I + q(-\Delta-z^2)^{-1}\chi)^{-1}$.
\end{de}
This definition is independent of the specific choice of $\chi$ (so long as $\chi\equiv 1$ on $\supp(q)$), and coincides with the poles of the scattering matrix of $q$, cf. \cite[Prop. 8]{Melrose} and \cite[III.5]{LP}.\\

Resonances can be regarded as states whose wave function disperses very slowly in time, and can therefore be considered as ``almost bound states''. 
In physics, such phenomena arise in the description of unstable particles and radioactive decay. Resonant states, just like eigenfunctions, can only exist at  certain energies. The slow-dispersal-in-time approach to resonances motivates one of the earlier definitions of resonances used in the computational physics literature, namely maximization of the so-called time delay function -- see, e.g., Le Roy and Liu \cite{LRL} and Smith \cite{Smith}. This approach leads to real resonance energies for real-valued potentials and, in the one-dimensional case at least, is closely related to the concept of spectral concentration -- see, e.g., Eastham \cite{MSPE}, which describes one mechanism by which such concentrations may arise. For additional discussion we refer to the review article \cite{Zworski1999} and the book \cite{DZ}.

It is widely accepted that the  reliable computation of resonances is a challenging task. This is not usually due to the intrinsic ill-posedness of analytic continuation, since that step is usually done explicitly, but rather due to the fact that complex scaling changes resonance problems either into non-selfadjoint spectral problems, for which the pseudospectra may be far from the spectrum \cite{ET},  or into problems with a nonlinear dependence on the spectral parameter, for which sensitivity to perturbations may also be problematic.  In this context we refer to \cite{Bogli2014}  (including the references and discussion therein) where interval-arithmetic was used to compute resonances.

We show that resonances can be computed as the limit of a sequence of approximations, each of which can be computed precisely using finitely many arithmetic operations. The proof is constructive: we define an algorithm and prove its convergence. We emphasize that this \emph{single} algorithm is valid for \emph{any} Schr\"odinger operator $H_q$ as defined above, so long as   $\diam(\supp(q))$  satisfies an \emph{a priori} bound. We  implement this algorithm in one-dimension and compare its performance to that of Bindel and Zworski \cite{Bindel}.

 \subsection{The Solvability Complexity Index}\label{subsec:sci}
The Solvability Complexity Index (SCI) addresses questions which are at the nexus of pure and applied mathematics, as well as computer science: 
 	\begin{quote}\emph{How do we compute objects that are ``infinite'' in nature if we can only handle a finite amount of information and perform finitely many mathematical operations? Indeed, what do we even mean by ``computing'' such an object?}
	\end{quote}
These broad topics are addressed in the sequence of papers \cite{Hansen11,AHS,Ben-Artzi2015a}. Let us summarize the main definitions and discuss how these relate to our problem of finding resonances:
\begin{de}[Computational problem]\label{def:computational_problem}
	A \emph{computational problem} is a quadruple $(\Om,\Lambda,\Xi,\mathcal M)$, where 
	\begin{enumi}
		\item $\Om$ is a set, called the \emph{primary set},
		\item $\Lambda$ is a set of complex-valued functions on $\Om$, called the \emph{evaluation set},
		\item $\mathcal M$ is a metric space,
		\item $\Xi:\Om\to \mathcal M$ is a map, called the \emph{problem function}.
	\end{enumi}
\end{de}

\begin{de}[Arithmetic algorithm]\label{def:Algorithm}
	Let $(\Om,\Lambda,\Xi,\mathcal M)$ be a computational problem. An \emph{arithmetic algorithm} is a map $\Gamma:\Om\to\mathcal M$ such that for each $T\in\Om$ there exists a finite subset $\Lambda_\Gamma(T)\subset\Lambda$ such that
	\begin{enumi}
		\item the action of $\Gamma$ on $T$ depends only on $\{f(T)\}_{f\in\Lambda_\Gamma(T)}$,
		\item for every $S\in\Om$ with $f(T)=f(S)$ for all $f\in\Lambda_\Gamma(T)$ one has $\Lambda_\Gamma(S)=\Lambda_\Gamma(T)$,
		\item the action of $\Gamma$ on $T$ consists of performing only finitely many arithmetic operations on $\{f(T)\}_{f\in\Lambda_\Gamma(T)}$.
	\end{enumi}
\end{de}

\begin{de}[Tower of arithmetic algorithms]\label{def:Tower}
	Let $(\Om,\Lambda,\Xi,\mathcal M)$ be a computational problem. A \emph{tower of algorithms} of height $k$ for $\Xi$ is a family $\Gamma_{n_1,n_2,\dots,n_k}:\Om\to\mathcal M$ of arithmetic algorithms such that for all $T\in\Om$
	\begin{align*}
		\Xi(T) = \lim_{n_k\to\infty}\cdots\lim_{n_1\to\infty}\Gamma_{n_1,n_2,\dots,n_k}(T).
	\end{align*}
\end{de}

\begin{de}[SCI]
	A computational problem $(\Om,\Lambda,\Xi,\mathcal M)$ is said to have a \emph{Solvability Complexity Index ($\SCI$)} of $k$ if $k$ is the smallest integer for which there exists a tower of algorithms of height $k$ for $\Xi$.
	If a computational problem has solvability complexity index $k$, we write \begin{align*}
 			\SCI(\Om,\Lambda,\Xi,\mathcal M)=k.
		 \end{align*}
\end{de}

In the present article, our computational problem is made up of the following elements (to be specified more precisely in Section \ref{sec:Setting}):
	\begin{enumi}
		\item $\Om$ is a class of Schr\"odinger operators $H_q$ with potentials $q$ which have a common compact support and a uniform bound in $\mathcal C^1$,
		\item $\Lambda$ is the set of all pointwise evaluations of $q$, as well as pointwise evaluations of the Green's function associated with the Helmholtz operator $-\Delta-z^2$,
		\item $\mathcal M$ is the space $\mathrm{cl}(\C)$ of all closed subsets of $\C$ equipped with the Attouch-Wets metric (which is a generalization of the Hausdorff metric to the case of unbounded sets),
		\item $\Xi:\Om\to \mathcal M$ is the map that associates to a particular Schr\"odinger operator its set of resonances, and we denote it by $\Res(H_q)$.
	\end{enumi}

We show that for this computational problem there exists a tower of height $1$, i.e. there exists a family of algorithms $\{\Gamma_n\}_{n\in\N}$ such that $\Gamma_n(H_q)\to\Res(H_q)$ as $n\to+\infty$ for any $H_q\in\Omega$, where the convergence is in the sense of the Attouch-Wets metric \cite{beer}, generated by the following distance function:

\begin{de}[Attouch-Wets distance] Let $A,B$ be closed sets in $\C$. The  \emph{Attouch-Wets distance} between them is defined as
\begin{align*}
	d_{\text{AW}}(A,B) = \sum_{i=1}^\infty 2^{-i}\min\left\{ 1\,,\,\sup_{|x|<i}\left| \dist(x,A) - \dist(x,B) \right| \right\}.
\end{align*} 
Note that if $A,B\subset\C$ are bounded, then $d_{\text{AW}}$ is equivalent to the Hausdorff distance.
\end{de}
\begin{remark}\label{remark:Attouch-Wets}
	It can be shown  (cf. \cite[Prop. 2.8]{R19}) that a sequence of sets $A_n\subset \C$ converges to $A$ in Attouch-Wets metric, if the following two conditions are satisfied
\begin{itemize}
	\item If $\lambda_n\in A_n$ and $\lambda_n\to\lambda$, then $\lambda\in A$.
	\item If $\lambda\in A$, then there exist $\lambda_n\in A_n$ with $\lambda_n\to\lambda$.
\end{itemize}

\end{remark}

\subsection{Main Result}\label{sec:Setting}

Let $d\in\mathbb N$, fix $M>0$ and let $\Om_M$ denote the class of Schr\"odinger operators
\begin{align*}
	H_q:=-\Delta+q\quad\text{on}\quad L^2(\R^d)
\end{align*}
with $q\in \mathcal C^1(\R^d;\C)$ and $\supp(q)\Subset Q_M$, where $Q_M$ denotes the cube of edge length $M$ centered at the origin.
Moreover, for $x\in\R^d$ and $z\in\C$ let 
\begin{align}\label{eq:Fundamental_Solution}
	G(x,z):=\begin{cases}\f\i 4\left( \f{z}{2\pi |x|} \right)^{\f{d-2}{2}}H^{(1)}_{\f{d-2}{2}}\big( z|x| \big), &d\ge 2,\\
		\f{\i}{2z}e^{\i z|x|}, & d=1,
		\end{cases}
\end{align}
where $H^{(1)}_{\nu}$ denotes the Hankel function of the first kind. For $\im(z)>0$, $G(x,z)$ is the fundamental solution to the free Helmholtz operator $-\Delta-z^2$ (cf. \cite[Ch. 22]{Serov}).
We define the evaluation set $\Lambda$ to be
\begin{align}\label{eq:input_info}
	\Lambda &:=  \left\{M \right\}\cup\left\{ q\mapsto q(x)\,|\, x\in\R^d \right\} 
		\cup \left\{  G(x,z)\,|\, x\in\R^d,\,z\in\C \right\}.
\end{align}
Then the quadruple $(\Om_M,\Lambda,\Res(\cdot),\text{cl}(\C))$ poses a computational problem in the sense of Definition \ref{def:computational_problem}. The main result of the present article is the following.
\begin{theorem}\label{th:mainth}
Resonances can be computed in one limit: $\SCI(\Om_M,\Lambda,\Res(\cdot),\mathrm{cl}(\C))=1.$
\end{theorem}
We  prove this theorem  by explicitly constructing an algorithm which computes the set of resonances in one limit. This algorithm can be implemented numerically; some numerical experiments are provided  in Section \ref{sec:Numerics}.
\begin{remark}
\begin{enumi}
	\item We note that there exist examples of computational spectral problems for which $\SCI\geq2$, even in the selfadjoint case  (cf. \cite[Th. 6.5]{AHS}). Whenever $\SCI\geq2$ it is impossible  to have error control (this can be shown using a straightforward diagonal-type argument). Identifying situations in which $\SCI=1$ is therefore of particular interest.
	\item If we drop the restriction $\supp(q)\Subset Q_M$ from the definition of $\Omega$, i.e. if the size of $\supp(q)$ is not assumed to be known \emph{a priori}, then Theorem \ref{th:mainth} implies $\SCI\leq 2$, but it is no longer clear that  $\SCI =1$. Indeed, let us provide a sketch of the proof:
		
		\emph{Sketch of proof:}  Choose a sequence $0\leq M_k\to+\infty$ and for each fixed $k$ let $\{\Gamma_{k,n}\}_{n\in\N}$ be a sequence of algorithms as in Theorem \ref{th:mainth}, which computes the resonances of $\Om_{M_k}$ in one limit. Given a compactly supported potential function $q$, whose support is not known, choose a smooth cutoff function $\rho_k$ with $Q_{M_k-1}\subset\supp(\rho_k)\Subset Q_{M_k}$. Then from Theorem \ref{th:mainth} we know that $\{\Gamma_{k,n}\}_{n\in\N}$ computes the set of resonances of $H_{\rho_k q}$ in one limit. As soon as $\supp(q)\Subset Q_{M_k-1}$, one has $\rho_k q\equiv q$ and the sequence $\{\lim_{n\to\infty}\Gamma_{nk}(H_{\rho_kq})\}_{k\in\N}$ will be constant in $k$, and hence convergent. This proves that $\lim_{k\to\infty}\lim_{n\to\infty}\Gamma_{k,n}(q)=\Res(H_q)$, that is, the set of resonances can be computed in two limits.
\end{enumi}
\end{remark}
The proof of Theorem \ref{th:mainth} is divided into several steps. First, we obtain quantitative resolvent norm estimates for the operator $K(z):=q(H_q-z^2)^{-1}\chi$ from Definition \ref{def:res}. These are then used to bound the error between $K(z)$ itself and a discretized version $K_n(z)$, obtained by replacing the potential $q$ by a piecewise constant approximation. Finally, the poles of $(I+K(z))^{-1}$ are identified through a thresholding of the discretized operator function $(I+K_n(z))^{-1}$.

\subsubsection*{Organization of the paper}

Section \ref{sec:Analytic_Continuation} contains a short discussion of Definition \ref{def:res} and meromorphic continuation. In Section \ref{sec:abstract} we prove some estimates for convergence of finite-dimensional approximations of linear operators, which are then used in Section \ref{sec:algorithm} to construct an explicit algorithm which computes resonances in one limit, thereby proving Theorem \ref{th:mainth}. Section  \ref{sec:Numerics} is dedicated to numerical experiments. In Appendix \ref{app:greens} we review some properties of the Green's function $G(x,z)$ introduced in \eqref{eq:Fundamental_Solution}.

\section{Analytic Continuation}\label{sec:Analytic_Continuation}
We  use this section for a more detailed discussion of Definition \ref{def:res} and to fix some notations and conventions.
For the sake of self containedness, we  prove the existence of $z\mapsto (I+q(-\Delta-z^2)^{-1}\chi)^{-1}$ as a meromorphic operator-valued function on the domain
\begin{align*}
	\C^{\mathrm{ext}}:=\begin{cases}
		\C & \text{if } d \text{ is odd},\\
		\text{logarithmic cover of }\C & \text{if } d \text{ is even}.
	\end{cases}
\end{align*}
This result follows from the classical Analytic Fredholm Theorem (cf. e.g. \cite[Sec. VI.5]{RS})
\begin{theorem}[Analytic Fredholm Theorem]
	Let $D\subset \C$ be open and connected and let $F:D\to L(\h)$ be an analytic operator-valued function such that $F(z)$ is compact for all $z\in D$. Then, either
	\begin{enumi}
		\item $(I+F(z))^{-1}$ exists for no $z\in D$, or
		\item $(I+F(z))^{-1}$ exists for all $z\in D\setminus S$, where $S$ is a discrete subset of $D$. In this case, $z\mapsto (I+F(z))^{-1}$ is meromorphic in $D$, analytic in $D\setminus S$, the residues at the poles are finite rank operators, and if $z\in S$ then $\ker(I+F(z))\neq\{0\}$.
	\end{enumi}
\end{theorem}

Next, recall that $Q_M$ denotes the cube of edge length $M$ in $\R^d$ centered at the origin. Let $\chi:=\chi_{Q_M}$ be the indicator function of $Q_M$. Note that the operator-valued function $z\mapsto q(-\Delta-z^2)^{-1}\chi$ is an analytic function on $\C^{\mathrm{ext}}\setminus\{0\}$. This follows from the explicit representation of the free fundamental solution \eqref{eq:Fundamental_Solution} (cf. Remark \ref{remark:even_odd}).
 
\begin{lemma}
	The function $\C^+\ni z\mapsto\big(I+q(-\Delta-z^2)^{-1}\chi\big)^{-1}$ has a meromorphic continuation to $\C^{\mathrm{ext}}$. 
	Moreover, the residues at the poles are finite rank operators.
\end{lemma}
\begin{proof}
The operator $q(-\Delta-z^2)^{-1}\chi$ is compact by the Fr\'echet-Kolmogorov theorem and the inverse $\big(I+q(-\Delta-z^2)^{-1}\chi\big)^{-1}$ exists for $\im(z)>0$ large enough, by the Neumann series. Hence, the claim follows from the analytic Fredholm theorem, together with Remark \ref{remark:even_odd} in the appendix.
\end{proof}

The above observations lead us to study the spectrum of the compact operator
\begin{align}\label{eq:free_resolvent}
	K(z):=q(-\Delta-z^2)^{-1}\chi, \qquad z\in\mathbb \C^{\mathrm{ext}}.
\end{align}
Since the integral kernel for the free resolvent is given explicitly by \eqref{eq:Fundamental_Solution} as an analytic function of $z\in\C^{\text{ext}}\setminus\{0\}$, we have an explicit representation of \eqref{eq:free_resolvent} as an integral operator on $L^2(\R^d)$:
\begin{align}\label{eq:integral_formula}
	\big(q(-\Delta-z^2)^{-1}\chi\, f\big)(x) = q(x)\int_{\R^d} G(x-y,z)\chi(y)f(y)\,dy, \qquad z\in \C^{\mathrm{ext}}\setminus\{0\}.
\end{align}

\section{Abstract Error Estimates}\label{sec:abstract}
We recall that  the resonances of $H_q=-\Delta+q$ are defined to be  the poles of $\C^{\mathrm{ext}}\ni z\mapsto\big(I+K(z)\big)^{-1}$ where $K(z)=q(-\Delta-z^2)^{-1}\chi$ is a compact operator. In this section we prove general, abstract, estimates for approximations of  families of linear operators. These are largely independent of the rest of this paper and will be applied in the proof of Theorem \ref{th:mainth}. Abusing notation, our generic abstract analytic operator family is denoted $K(z)$.

Let $\h$ be a separable Hilbert space and denote by $L(\h)$ the space of bounded operators on $\h$. Let $\h_n\subset \h$ be a finite-dimensional subspace, $P_n:\h\to\h_n$ the orthogonal projection  and $K:\C^{\mathrm{ext}}\to L(\h)$ continuous in operator norm.  Moreover, let $K_n:\C^{\mathrm{ext}}\to L(\h_n)$ be analytic for every $n\in\N$. Assume that for any compact subset $B\subset \C^\mathrm{ext}$ there exist a sequence $a_n\downarrow0$ and a constant $C>0$ such that
\begin{align}
	\|K(z)-K_n(z)P_n\|_{L(\h)} &\leq Ca_n, \label{eq:K-KnPn}\\
	\left\|P_nK(z)|_{\h_n}-K_n(z)\right\|_{L(\h_n)} &\leq Ca_n, \label{eq:PnK|-Kn}\\
	\|K(z)-P_nK(z)P_n\|_{L(\h)} &\leq Ca_n,\label{eq:K-PKP}
\end{align}
for all $z\in B$. 
\subsection{Error Estimates}
\begin{lemma}\label{lemma:pollution}
	If $z\in\C^\mathrm{ext}$ is such that $-1\notin\sigma(K(z))$, then 
	\begin{align*}
		\big(1-Ca_n\|(I+K(z))^{-1}\|_{L(\h)}\big)\left\|(I+K_n(z))^{-1}\right\|_{L(\h_n)} \leq \left\|(I+K(z))^{-1}\right\|_{L(\h)},
	\end{align*}
	where we use the convention that $\|(I+K_n(z))^{-1}\|_{L(\h)}=+\infty$ if $-1\in\sigma(K_n(z))$.
\end{lemma}
\begin{proof}
	Whenever the left hand side is non-positive the assertion is trivially true, so we may assume w.l.o.g. that $1-Ca_n\|(I+K(z))^{-1}\|_{L(\h)}>0$.
	In this case, the assertion follows by a Neumann series argument, as follows. We have
	\begin{equation}\label{eq:factor}
	\begin{aligned}
		I+K_n(z)P_n &= I+K(z)+(K_n(z)P_n-K(z))\\
		&= (I+K(z))\left[ I+(I+K(z))^{-1}(K_n(z)P_n-K(z)) \right]
	\end{aligned}
	\end{equation}
	Because $Ca_n<\f{1}{\|(I+K(z))^{-1}\|}$, the second factor in \eqref{eq:factor} is invertible by the Neumann series and
	\begin{align*}
		\left[ I+(I+K(z))^{-1}(K_n(z)P_n-K(z)) \right]^{-1} = \sum_{j=0}^\infty \Big((I+K(z))^{-1}(K_n(z)P_n-K(z))\Big)^j.
	\end{align*}
	Hence,
	\begin{align*}
		\left\|(I+K_n(z)P_n)^{-1}\right\|_{L(\h)} &\leq \Bigg\| \sum_{j=0}^\infty \Big((I+K(z))^{-1}(K_n(z)-K(z))\Big)^j \Bigg\|_{L(\h)} \left\|(I+K(z))^{-1}\right\|_{L(\h)} \\
		&\leq \sum_{j=0}^\infty \left\|(I+K(z))^{-1}\right\|_{L(\h)}^{j+1} \left\|K_n(z)P_n-K(z)\right\|_{L(\h)}^j\\
		&\leq \sum_{j=0}^\infty \left\|(I+K(z))^{-1}\right\|_{L(\h)}^{j+1} (Ca_n)^j\\
		&= \left\|(I+K(z))^{-1}\right\|_{L(\h)} \sum_{j=0}^\infty \left\|(I+K(z))^{-1}\right\|_{L(\h)}^{j} (Ca_n)^j\\
		&= \f{\|(I+K(z))^{-1}\|_{L(\h)}}{1-\left\|(I+K(z))^{-1}\right\|_{L(\h)} Ca_n}
	\end{align*}	
for any $n\in\N$.	It remains to replace the $L(\h)$ norm on the left hand side by the $L(\h_n)$ norm. This follows from Claim \ref{claim:replace-norm}. 	This completes the proof.
\end{proof}
\begin{cl}\label{claim:replace-norm}
	 We have $\|(I+K_n(z))^{-1}\|_{L(\h_n)}\leq \|(I+K_n(z)P_n)^{-1}\|_{L(\h)}$ for all $z$ for which both operators are boundedly invertible.
\end{cl}
	\begin{proof}
	For $x\in\h_n$ we have $(I+K_nP_n)^{-1}x = (I+K_n)^{-1}x$, because if $u\in\h_n$ solves $(I+K_n)u=x$, then $(I+K_nP_n)u=x$ and by invertibility it follows that $u=(I+K_nP_n)^{-1}x$. We conclude that 
		\begin{align*}
			\sup_{x\in\h_n,\|x\|=1} \|(I+K_nP_n)^{-1}x\|_{\h} = \sup_{x\in\h_n,\|x\|=1} \|(I+K_n)^{-1}x\|_{\h_n}
		\end{align*}
		and therefore
		\begin{align*}
			\sup_{x\in\h,\|x\|=1} \|(I+K_nP_n)^{-1}x\|_{\h} \ge \sup_{x\in\h_n,\|x\|=1} \|(I+K_n)^{-1}x\|_{\h_n}.
		\end{align*}
\end{proof}

\begin{lemma}\label{lemma:inclusion}
	If $z\in\mathbb C^\mathrm{ext}$ is such that either $-1\in\sigma(K(z))$ or $\left\|(I+K(z))^{-1}\right\|_{L(\h)}\ge \f{1}{Ca_n}$, then either $-1\in\sigma(P_nK(z)P_n)$ or
	\begin{equation*}\label{eq:resolvent>1/3a}
		\left\|(I+P_nK(z)P_n)^{-1}\right\|_{L(\h)}\ge \f{1}{2Ca_n}.
	\end{equation*}
\end{lemma}
\begin{proof}
	If $-1\in\sigma(K(z))$, then unless $-1\in\sigma(P_nK(z)P_n)$, we have 
	\begin{equation*}\label{eq:factor2}
	\begin{aligned}
		I+K(z) &= I+P_nK(z)P_n+(K(z)-P_nK(z)P_n)\\
		&= (I+P_nK(z)P_n)\left[ I+(I+P_nK(z)P_n)^{-1}(K(z)-P_nK(z)P_n) \right]
	\end{aligned}
	\end{equation*}
	We now argue by contradiction. If we had $\|(I+P_nK(z)P_n)^{-1}\|_{L(\h)}<\f{1}{2Ca_n}$, then we would have $\|(I+P_nK(z)P_n)^{-1}(K(z)-P_nK(z)P_n)\|_{L(\h)}<1$ and $I+K(z)$ would be invertible by the Neumann series contradicting our assumption that $-1\in\sigma(K(z))$. Thus we must have $\|(I+P_nK(z)P_n)^{-1}\|_{L(\h)} \geq \f{1}{2Ca_n}$.
	
	Now let us turn to the case where $-1\notin\sigma(K(z))$ and $\left\|(I+K(z))^{-1}\right\|_{L(\h)}\ge \f{1}{Ca_n}$. The same calculation as in the proof of Lemma \ref{lemma:pollution} shows that 
	\begin{align*}
		\left(1-Ca_n\left\|(I+P_nK(z)P_n)^{-1}\right\|_{L(\h)}\right)\left\| (I+K(z))^{-1} \right\|_{L(\h)} \leq \left\|(I+P_nK(z)P_n)^{-1}\right\|_{L(\h)}
	\end{align*}
	from which it follows easily that $\f{1}{2Ca_n}\leq \left\|(I+P_nK(z)P_n)^{-1}\right\|_{L(\h)}$.
\end{proof}

\subsection{An Abstract Algorithm For Computing Poles}

We now demonstrate how the  the assumptions \eqref{eq:K-KnPn}-\eqref{eq:K-PKP} allow us to construct an abstract algorithm that computes the poles of  $\big(I+K(z)\big)^{-1}$. By an \emph{abstract algorithm} we mean a sequence of subsets of $\C^{\mathrm{ext}}$, which is constructed from $K_n$ and which converges in Attouch-Wets metric to $\{z\in \C^{\mathrm{ext}}\,|\,-1\in\sigma(K(z))\}$. Note that this is not yet an \emph{arithmetic algorithm} in the sense of Definition \ref{def:Algorithm}, since the sets are not computed from a finite amount of information in finitely many steps.

Let $B\subset\C^{\mathrm{ext}}$ be compact and define the exponentially fine lattice $\mathcal{L}_n:=e^{-\f1{a_n}}(\Z+\i\Z)\cap B$. Since we assume that $a_n$ is explicitly known and $K_n(z)$ can be computed in finitely many steps, we can define the set
\begin{align*}
	\Theta_n^B(K) = \left\{z\in \mathcal{L}_n\,\bigg|\,\left\|(I+K_n(z))^{-1}\right\|_{L(\h_n)} \ge \f{1}{2\sqrt{a_n}} \right\}.
\end{align*}
Moreover, note that by \cite[Prop. 10.1]{AHS}, determining whether $\left\|(I+K_n(z))^{-1}\right\|_{L(\h_n)} \ge \f{1}{2\sqrt{a_n}}$ can be done with finitely many arithmetic operations on the matrix elements of $K_n(z)$ for each $z\in\mathcal L_n$.
\begin{lemma}\label{lemma:convergence}
	The assumptions \eqref{eq:K-KnPn}-\eqref{eq:K-PKP} imply the convergence $\Theta_n^B(K)\to \{z\in B\,|\,-1\in\sigma(K(z))\}$ in Attouch-Wets metric.
\end{lemma}

\begin{proof}

\emph{I. Excluding spectral pollution.} Assume that $z_n\in \Theta_n^B(K)$ with $z_n\to z_0$ for some $z_0\in B$. Then for each $n$ we have $\|(I+K_n(z_n))^{-1}\|_{L(\h_n)} \ge \f{1}{2\sqrt{a_n}}$ and hence by Lemma \ref{lemma:pollution}
\begin{equation*}
\left\|(I+K(z_n))^{-1}\right\|_{L(\h)} \ge \Big(1-Ca_n\left\|(I+K(z_n))^{-1}\right\|_{L(\h)}\Big)\f12a_n^{-\f12}. 
\end{equation*}
(with the convention that $\|(I+K(z_n))^{-1}\|_{L(\h)}=+\infty$ if $-1\in\sigma(K(z_n))$). Whenever $\sqrt{a_n}\leq \f{2}{C}$ this leads to
\begin{align*}
	\left\|(I+K(z_n))^{-1}\right\|_{L(\h)} &\ge \f12\, \f{a_n^{-\f12}}{1+C\f{\sqrt{a_n}}{2}}\\
	&\ge \f14 a_n^{-\f12}.
\end{align*}
It follows  that $\|(I+K(z_n))^{-1}\|_{L(\h)} \to+\infty$ as $n\to+\infty$ and hence $I+K(z_0)$ is not invertible (this follows by yet another Neumann series argument, together with norm continuity of $K$). Hence $z_0$ is a pole.
	
	
\emph{II. Spectral inclusion.}
Assume now that $z$ is a pole, i.e. $-1\in\sigma(K(z))$. Our reasoning will have the structure
\begin{align*}
	-1\in & \sigma(K(z))\\
	&\Downarrow\\
	\exists z_n\in \mathcal{L}_n : \|(I&+K(z_n))^{-1}\|_{L(\h)} \;\text{ large} \\
	&\Downarrow\\
	\|(I+P_nK&(z_n) P_n)^{-1}\|_{L(\h)} \;\text{ large}\\
	&\Downarrow\\
	\|(I+P_nK&(z_n)|_{\h_n})^{-1}\|_{L(\h_n)} \;\text{ large}\\
	&\Downarrow\\
	\|(I+K_n&(z_n))^{-1}\|_{L(\h_n)} \;\text{ large,}
\end{align*}
with a quantitative estimate in each step. To this end, note first that if $-1\in\sigma(K(z))$ for some $z\in B$, then there exist $\nu,c,\eps>0$ (independent of $n$) such that for all  $\zeta$ in a $\eps$-neighborhood of $z$,
\begin{align}\label{eq:Laurent}
	\|(I+K(\zeta))^{-1}\|_{L(\h)}\geq c |z-\zeta|^{-\nu}.
\end{align}
Indeed, since all singularities of $(I+K(z))^{-1}$ are of finite order by the analytic Fredholm theorem, this follows from the Laurent expansion of meromorphic operator valued functions.

It follows from \eqref{eq:Laurent} that for any $z_n$ such that $|z-z_n|\leq e^{-\f1{a_n}}$ one will have 
\begin{equation*}
	\|(I+K(z_n))^{-1}\|_{L(\h)} \geq c |z-z_n|^{-\nu}
	\geq c e^{\f{\nu}{a_n}}
	\geq \f1{Ca_n}
\end{equation*}
for all large enough $n$. We conclude that for any pole $z$ there exists a sequence $z_n\in \mathcal{L}_n$ such that  $z_n\to z$ as $n\to+\infty$ and $\|(I+K(z_n))^{-1}\|_{L(\h)}>\f1{Ca_n}$ for  all but finitely many $n\in\N$.

Next, Lemma \ref{lemma:inclusion} shows that $\|(I+P_nK(z_n)P_n)^{-1}\|_{L(\h)}>\f{1}{2Ca_n}$. Studying this norm further, we have
\begin{align*}
	(I_\h+P_nK(z_n)P_n)^{-1} &= \bigl(I_{\h_n}+P_nK(z_n)|_{\h_n}\bigr)^{-1}\oplus I_{\h_n^\perp}
\end{align*}
and thus
\begin{align*}
	\left\|(I_\h+P_nK(z_n)P_n)^{-1}\right\|_{L(\h)} &= \max\left\{\bigl\|\bigl(I_{\h_n}+P_nK(z_n)|_{\h_n}\bigr)^{-1}\bigr\|_{L(\h_n)}\,,\, 1\right\}.
\end{align*}
Hence, as soon as $a_n<\f1{2C}$, we have $\|(I+P_nK(z_n)P_n)^{-1}\|_{L(\h)} = \|(I+P_nK(z_n)|_{\h_n})^{-1}\|_{L(\h_n)}$. We conclude that if $z$ is a pole, then there exists $z_n\in \mathcal{L}_n$ such that
\begin{align}\label{eq:PnK|_bound}
	\|(I+P_nK(z_n)|_{\h_n})^{-1}\|_{L(\h_n)}>\f1{2Ca_n}
\end{align}
($n$ large enough). A similar reasoning as in Lemma \ref{lemma:pollution} (using \eqref{eq:PnK|-Kn}) shows that now 
\begin{align*}
	\big(1-Ca_n\|(I+K_n(z_n))^{-1}\|_{L(\h_n)}\big)\|(I+P_nK(z_n)|_{\h_n})^{-1}\|_{L(\h_n)} \leq \|(I+K_n(z_n))^{-1}\|_{L(\h_n)},
\end{align*} 
and rearranging terms, together with \eqref{eq:PnK|_bound}, gives
\begin{align*}
	\|(I+K_n(z_n))^{-1}\|_{L(\h_n)} \geq \f{1}{4Ca_n}
\end{align*}
and therefore $z_n\in\Theta_n^B(K)$ for large enough $n$.
The assertion about Attouch-Wets convergence now follows from {Remark \ref{remark:Attouch-Wets}.}
\end{proof}

\section{Definition of the Algorithm}\label{sec:algorithm}

\subsection{Error Estimates}

In this section, we will apply the abstract results of Section \ref{sec:abstract} to our resonance problem. To this end, define $K(z) = q(-\Delta-z^2)^{-1}\chi$.
We write $K$ for the integral kernel of $K(z)$ to simplify notation. Recall that $K$ was given by \eqref{eq:integral_formula} and $\supp(K)\subset Q_M\times Q_M$. 
We will construct an operator approximation $K_n$ of $K$, which satisfies \eqref{eq:K-KnPn}-\eqref{eq:K-PKP} and in addition
\begin{itemize}
	\item[(H1)] The matrix elements of $K_n$ can be computed in finitely many steps from a finite subset $\Lambda_n\subset\Lambda$ (cf. eq. \eqref{eq:input_info} and Def. \ref{def:Algorithm});
	\item[(H2)] The convergence rate $a_n$ is explicitly known (i.e. the sequence $a_n$ can be used to define the algorithm).
\end{itemize}
To this end, let us define $\h_n,P_n$ as follows.
\begin{align}
	\R^d &= \bigcup_{i\in\f1n\Z^d} S_{n,i} :=\bigcup_{i\in\f1n\Z^d} \left(\big[0,\tfrac 1 n\big)^d+i\right) \label{eq:R^d_decomp}\\
	\h_n &= \left\{ f\in L^2(Q_M)\,\middle|\, f|_{S_{n,i}} \text{ constant }\forall i\in\tfrac1n\Z^d \cap Q_M \right\}\notag \\
	P_nf(x) &= \sum_{i\in\f1n\Z^d\cap Q_M} \bigg(n^d\int_{S_{n,i}} f(t)\,dt\bigg)\chi_{S_{n,i}}(x)\label{eq:P_n_def}
\end{align}
Furthermore, we have to make a concrete choice for the approximation $K_n$. An obvious choice is the integral kernel
\begin{align*}
	K_n(x,y) = \sum_{i,j\in n^{-1}\Z^d\cap Q_M} K(i,j)\chi_{S_{n,i}}(x)\chi_{S_{n,j}}(y),
\end{align*}
i.e. a piecewise constant approximation of $K(\cdot,\cdot)$ which can be computed from the values of $K$ on the lattice $n^{-1}\Z^d$ (in dimensions larger than one, the fundamental solution $G$ has a singularity at $x=y$. Hence, we put $K_n:=0$ for $i=j$ in this case).

We will now show that the operators $K,K_n$ satisfy eqs. \eqref{eq:K-KnPn}-\eqref{eq:K-PKP}. To streamline the presentation, we will restrict ourselves to $d\ge 3$ in our computations, the cases $d\leq 2$ being entirely analogous with minor changes in the formulas.
Constants independent of $n$ will be denoted $C$ and their value may change from line to line.

\paragraph{Proof of \eqref{eq:K-PKP}.}

Using the definitions \eqref{eq:R^d_decomp}-\eqref{eq:P_n_def}, we have
\begin{align*}
	Kf(x) - P_nKP_nf(x) &= \int_{\R^d} K(x,y)f(y)\,dy - \int_{\R^d} P_n^xK(x,y)P_nf(y)\,dy,
\end{align*}
where $P_n^xK(x,y)$ means $(P_nK(\cdot,y))(x)$. Using $L^2$-selfadjointness of $P_n$, we conclude that
\begin{align*}
	Kf(x) - P_nKP_nf(x) &= \int_{\R^d} K(x,y)f(y)\,dy - \int_{\R^d} P_n^yP_n^xK(x,y)f(y)\,dy\\
	 &= \int_{\R^d} \big(K(x,y)- P_n^yP_n^xK(x,y)\big)f(y)\,dy,
\end{align*}

Note that $P_n^yP_n^xK(x,y)$ simply yields a step function approximation of $K(x,y)$ like \eqref{eq:P_n_def}, but in dimension $2d$. We conclude by applying Young's inequality \cite[Th. 0.3.1]{Sogge}, that
\begin{align*}
	\|Kf - P_nKP_nf\|_{L^2(\R^d)} &\leq \eta_n \|f\|_{L^2(\R^d)},
\end{align*}
where
\begin{align}\label{eq:eta_n}
	\eta_n = \max\left\{ \sup_{x\in\R^d}\int_{\R^d} |K(x,y)- P_n^yP_n^xK(x,y)|\,dy \,,\,\sup_{y\in\R^d}\int_{\R^d} |K(x,y)- P_n^yP_n^xK(x,y)|\,dx \right\}
\end{align}
Thus, all we have to do is estimate the $L^\infty$-$L^1$ difference between $K$ and its projection onto step functions. To this end, fix $x\in Q_M$, let $\eps>\f{2}{n}$ and decompose the integrals as follows
\begin{align}\label{eq:K-PPK_whole}
\begin{aligned}
	\int_{\R^d} |K(x,y)- & P_n^yP_n^xK(x,y)|\,dy = \int_{Q_M} |K(x,y)-P_n^yP_n^xK(x,y)|\,dy  \\
	&= \int_{Q_M\setminus B_\eps(x)} |K(x,y)-P_n^yP_n^xK(x,y)|\,dy + \int_{B_\eps(x)} |K(x,y)-P_n^yP_n^xK(x,y)|\,dy 
\end{aligned}
\end{align}
The integral over $B_\eps(x)$ can be estimated by $\int_{B_\eps(x)} 2|K(x,y)| \,dy$, while for the remaining integral we can use the fact that the derivative of $K$ is bounded, as follows. Let $j\in\f1n\Z^d$ be such that $x\in S_{n,j}$, see Figure \ref{fig:Sketch}. Let  $i\in\f1n\Z^d$ be such that $|i-j|>\f{\eps}{2}$. Then:

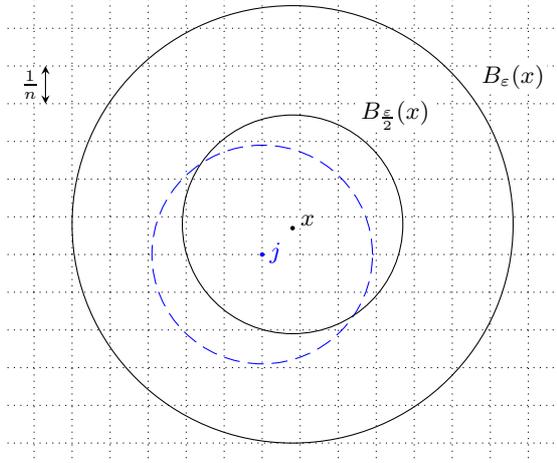
\begin{figure}[htbp]
	\centering
%
%

\begin{tikzpicture}[>=stealth,scale=0.5]
	\filldraw[color=blue, fill=white, dash pattern={on 6pt off 3pt}] (-0.8,-0.8) circle(2.9);
	\draw[color=black, dotted, shift={(0.2,0.2)}] (-7.7,-6.4) grid (6.9,5.9);
	\draw[color=black] (0,0) circle(2.9);
	\fill[fill=black](0,-0.1) circle(0.06);
	\node[fill=white,inner sep=1pt] at (0.4,0.1) {\small$x$};
	\node[fill=white,inner sep=1pt] at (2.7,2.9) {\small$B_{\frac{\varepsilon}{2}}(x)$};
	
	\fill[fill=blue](-0.8,-0.8) circle(0.06);
	\node[fill=white,inner sep=1pt] at (-0.45,-0.75) {\small\color{blue}$j$};
	
	\draw[color=black] (0,0) circle(5.8);
	\node[fill=white,inner sep=1pt] at (5.8,3.9) {\small$B_{\varepsilon}(x)$};
	\node[fill=white,inner sep=1pt] at (-6.9,3.7) {\small$\tfrac1n$};
	\draw[<->]  (-6.5,3.2) -- (-6.5,4.2);
\end{tikzpicture}

	\caption{
		Sketch of the geometry in the calculation leading to \eqref{eq:complement_of_B_eps}. The sum over $i$ includes all cells whose nodes are outside the dashed ball centered at $j$. 
	}
	\label{fig:Sketch}
\end{figure}

\begin{align*}
	\int_{S_{n,i}} |K(x,y) & -P_n^yP_n^xK(x,y)|\,dy = \int_{S_{n,i}} \left|K(x,y)-\dashint_{S_{n,i}\times S_{n,j}} K(s,t)dsdt\right|\,dy\\
	&\leq \int_{S_{n,i}} \dashint_{S_{n,i}\times S_{n,j}} \left|K(x,y)- K(s,t)\right|dsdt\,dy\\
	&= \int_{S_{n,i}} \dashint_{S_{n,i}\times S_{n,j}} \left|\int_0^1\nabla K\big(\tau\big(\begin{smallmatrix}x\\y\end{smallmatrix}\big)+(1-\tau)\big(\begin{smallmatrix}s\\t\end{smallmatrix}\big)\big)\cdot\big(\big(\begin{smallmatrix}x\\y\end{smallmatrix}\big)-\big(\begin{smallmatrix}s\\t\end{smallmatrix}\big)\big)d\tau \right|dsdt\,dy\\
	&\leq \int_{S_{n,i}} \dashint_{S_{n,i}\times S_{n,j}} \int_0^1\big|\nabla K\big(\tau\big(\begin{smallmatrix}x\\y\end{smallmatrix}\big)+(1-\tau)\big(\begin{smallmatrix}s\\t\end{smallmatrix}\big)\big)\big|\big|\big(\begin{smallmatrix}x\\y\end{smallmatrix}\big)-\big(\begin{smallmatrix}s\\t\end{smallmatrix}\big)\big| d\tau\, dsdt\,dy\\
	&\leq \int_{S_{n,i}} \dashint_{S_{n,i}\times S_{n,j}} \int_0^1\big\|\nabla K\big\|_{L^\infty(S_{n,i}\times S_{n,j})}\f{2\sqrt d}n d\tau\, dsdt\,dy\\
\end{align*}
Summing over $i$, we finally obtain (cf. Figure \ref{fig:Sketch})
\begin{align}
	\int_{\R^d\setminus B_{\eps}(x)} |K(x,y)-P_n^yP_n^x K(x,y)|\,dy &\leq \sum_{i:|i-j|>\f\eps2} \int_{S_{n,i}} |K(x,y)-P_n^yP_n^xK(x,y)|\,dy\nonumber\\
	&\leq \sum_{i:|i-j|>\f\eps2} \int_{S_{n,i}} \dashint_{S_{n,i}\times S_{n,j}} \int_0^1\big\|\nabla K\big\|_{L^\infty(S_{n,i}\times S_{n,j})}\f{2\sqrt d}n d\tau\, dsdt\,dy\nonumber\\
	&\leq \f{2\sqrt d}n \big\|\nabla K\big\|_{L^\infty(Q_M\setminus B_{\f{\eps}{2}}(x))} \int_{Q_M\setminus B_{\f{\eps}{4}}(x)} \,dy\nonumber\\
	&= |Q_M|\f{2\sqrt d}{n} \big\|\nabla K\big\|_{L^\infty(Q_M\setminus B_{\f{\eps}{2}}(x))} \nonumber\\
	&\leq  |Q_M|\f{2\sqrt d}{n} \|q\|_{\mathcal C^1} C\left(\f{\eps}{2}\right)^{1-d}\nonumber\\
	&\leq C\f{|Q_M|}n \eps^{1-d}, \label{eq:complement_of_B_eps}
\end{align}
where the fifth line follows from \eqref{eq:G_bound_2}, in the appendix, and the bound $\|q\|_{\mathcal C^1}\leq +\infty$. Using \eqref{eq:complement_of_B_eps} in \eqref{eq:K-PPK_whole}, we conclude that
\begin{align*}
	\int_{\R^d} |K(x,y)- P_n^yP_n^xK(x,y)|\,dy &\leq C\f{|Q_M|}{n} \eps^{1-d} + \int_{B_\eps(x)} 2|K(x,y)|\,dy\\
	\Rightarrow\quad\sup_{x\in \R^d}\int_{\R^d} |K(x,y)- P_n^yP_n^xK(x,y)|\,dy &\leq C\f{|Q_M|}{n} \eps^{1-d} + C'\eps^2,
\end{align*}
where in the last line we have used \eqref{eq:G_bound} and the boundedness of $q$ again.

With an analogous calculation for $\sup_{y\in\R^d}\int_{\R^d} |K(x,y)- P_n^yP_n^xK(x,y)|\,dx$ (which we omit here), and recalling that $\eta_n$ was defined by \eqref{eq:eta_n}, we conclude that for all $\eps>0$
\begin{align*}
	\eta_n\leq \f{1}{n}C\eps^{1-d} + C'\eps^2.
\end{align*}
Choosing $\eps:=n^{-\f{1}{d+1}}$, we conclude that
\begin{align}\label{eq:explicit_rate}
	\|Kf - P_nKP_nf\|_{L^2(\R^d)} &\leq \f{C+C'}{n^{\f{2}{d+1}}} \|f\|_{L^2(\R^d)}
\end{align}
and hence $\|K-P_nKP_n\|_{L(L^2(\R^d))}\to 0$ as $n\to+\infty$ with rate (at least) $a_n=n^{-\f{2}{d+1}}\leq n^{-\f1d}$.

\begin{remark}
	Note that the constants $C,C'$ all depend on the spectral parameter $z$, but are bounded for $z$ in compact subsets of $\C^{\mathrm{ext}}$, because $K$ depends continuously on $z$.
\end{remark}

\paragraph{Proof of \eqref{eq:PnK|-Kn} and (H1).} 

An orthonormal basis of $\h_n$ is given by the functions 
\begin{align*}
	e_i:=n^{\f d2}\chi_{S_{n,i}},\qquad i\in\tfrac1n\Z^d\cap Q_M,
\end{align*}
so that
	\begin{equation*}
	P_nf = \sum_{j\in \f1n\Z\cap Q_M}\p{f,e_j}_{L^2}e_j
	\end{equation*}
in this basis. It is then easily seen that in this basis $K_n$ has the matrix elements
\begin{equation*}
	(K_n)_{ij} = n^{-d} K(i,j).
\end{equation*}
Note that this proves (H1): The matrix elements of $K_n$ can be calculated in finitely many arithmetic operations from the finite set $\Lambda_n:=\{K(i,j)\,|\,i,j\in\f1n\Z\cap Q_M\}\subset\Lambda$.
Similarly, it can be seen that the matrix elements of $P_nK|_{\h_n}$ in this basis are given by
\begin{align*}
	(P_nK)_{ij} &= n^{d}\int_{S_{n,i}}\int_{S_{n,j}}K(x,y)\,dxdy  \\[1mm]
	&=: n^{-d} \langle K\rangle_{ij},
\end{align*}
where we have introduced the notation $\langle \cdot\rangle_{ij}$ for the mean value on $S_{n,i}\times S_{n,j}$.
Let $f=\sum_{j} f_je_j \in \h_n$. From the above, and Young's inequality, we conclude that
\begin{align*}
	\|(P_nK-K_n)f\|_{L^2}^2 &= \sum_{i\in \f1n\Z^d\cap Q_M}\left|\sum_{j\in \f1n\Z^d\cap Q_M} n^{-d} \big(K(i,j)-\langle K\rangle_{ij}\big)f_j\right|^2\\
	&\leq \tilde\eta_n^2\|f\|_{L^2}^2,
\end{align*}
where 
\begin{align*}
		\tilde\eta_n := \max\Bigg\{ \sup_{i\in\f1n\Z^d\cap Q_M}\,\sum_{j\in \f1n\Z^d\cap Q_M} n^{-d}|K(i,j)-\langle K\rangle_{ij}| \,,\,\sup_{j\in\f1n\Z^d\cap Q_M}\,\sum_{i\in \f1n\Z^d\cap Q_M} n^{-d}|K(i,j)-\langle K\rangle_{ij}| \Bigg\}.
\end{align*}
Hence, we have reduced the problem to estimating these $\ell^\infty$-$\ell^1$ differences. This can be done similarly to \eqref{eq:K-PPK_whole}, by separating $(Q_M\times Q_M)\cap (\f1n\Z\times\f1n\Z)$ into an $\eps$-region around $i=j$ and the rest: 
\begin{align}
	\sum_{j\in \f1n\Z^d\cap Q_M} n^{-d}|K(i,j)-\langle K\rangle_{ij}| &= \sum_{|j-i|>\eps}n^{-d}|K(i,j)-\langle K\rangle_{ij}| + \sum_{|j-i|\leq\eps}n^{-d}|K(i,j)-\langle K\rangle_{ij}| \nonumber \\
	&\leq Cn^{-1}\sum_{|j-i|>\eps} n^{-d} \|\nabla K\|_{L^\infty(\{|x-y|>\eps\})} + \sum_{|j-i|\leq\eps}n^{-d}|K(i,j)-\langle K\rangle_{ij}|  \nonumber\\
	&\leq Cn^{-1} \eps^{-d+1} + \sum_{|j-i|\leq\eps}n^{-d}|K(i,j)-\langle K\rangle_{ij}|,  \label{eq:discrete_estimate}
\end{align}
where we have used \eqref{eq:G_bound_2} and the $\mathcal C^1$-boundedness of $q$ in the last line. To estimate the last term on the right hand side, note that $|K(i,j)-\langle K\rangle_{ij}|\leq C|j-i|^{-(d-2)}$ near $i=j$ (cf. eq. \eqref{eq:G_bound}). 
Next, note that the sum $n^{-d}\sum_{j:|j-i|\leq\eps} \f{1}{|j-i|^{d-2}}$ can be interpreted as an integral over a piecewise constant function, which approximates $(x,y)\mapsto |x-y|^{2-d}$. But this function is dominated by $(x,y)\mapsto |x-y|^{1-d}$ when $|x-y|$ is small, and therefore we have
\begin{align}
	n^{-d}\sum_{j:|j-i|\leq\eps} \f{1}{|j-i|^{d-2}} &\leq C \int_{B_{2\eps}(x)} |x-y|^{1-d}\,dy \nonumber\\
	&= C\int_0^{2\eps}  r^{1-d}\,\omega_d r^{d-1} dr \nonumber\\
	&= 2C\omega_d\, \eps \label{eq:discrete_2}
\end{align}
where $\omega_d$ denotes the volume of the unit sphere in $\R^d$.
Note that the above calculation is uniform in $i$, because $q$ is bounded. 
Plugging \eqref{eq:discrete_2} into \eqref{eq:discrete_estimate}, we arrive at
\begin{align*}
	\sum_{j\in \f1n\Z^d\cap Q_M} n^{-d}|K(i,j)-\langle K\rangle_{ij}| 
	&\leq Cn^{-1} \eps^{-d+1} + 2C\omega_d\, \eps.
\end{align*}
Choosing $\eps=n^{-\f1d}$ yields
\begin{align}\label{eq:explicit_rate2}
	\sum_{j\in \f1n\Z^d\cap Q_M} n^{-d}|K(i,j)-\langle K\rangle_{ij}| 
	&\leq C'n^{-\f1d}.
\end{align}
 Finally, swapping $i$ and $j$ will give an analogous estimate and we can conclude that $\tilde\eta_n\to 0$ with rate $a_n = n^{-\f1d}$.
\begin{remark}
	Note again that the constants $C,C'$ depend on $z$, but are bounded for $z$ in compact subsets of $\C^{\mathrm{ext}}$, since $K$ depends continuously on $z$.
\end{remark}

\paragraph{Proof of \eqref{eq:K-KnPn} and (H2).}
Estimate \eqref{eq:K-KnPn} in fact follows from \eqref{eq:K-PKP} and \eqref{eq:PnK|-Kn}. Indeed, writing $K_n$ and $K$ as block operator matrices w.r.t. the decomposition $\h = \h_n\oplus\h_n^\perp$, we have
\begin{align*}
	K = \begin{pmatrix}
		P_nK|_{\h_n} & D_1\\
		D_2 & D_3
	\end{pmatrix},
\end{align*}
with some operators $D_1,D_2,D_3$. Estimate \eqref{eq:K-PKP} shows that 
\begin{align}\label{eq:0BCD}
	\left\|\begin{pmatrix}
		0 & D_1\\
		D_2 & D_3
	\end{pmatrix}\right\|_{L(\h)}<Ca_n,
\end{align}
whereas estimate \eqref{eq:PnK|-Kn} shows that 
\begin{align}\label{eq:K000}
	\left\|P_nK|_{\h_n}-K_n\right\|_{L(\h_n)} = \left\|\begin{pmatrix}
		P_nK|_{\h_n}-K_n & 0\\
		0 & 0
	\end{pmatrix}\right\|_{L(\h)}<Ca_n.
\end{align}
Together, eqs. \eqref{eq:0BCD} and \eqref{eq:K000} imply that
\begin{align*}
	\|K(z)-K_n(z)P_n\|_{L(\h)}  = \left\|\begin{pmatrix}
		P_nK|_{\h_n}-K_n & D_1\\
		D_2 & D_3
	\end{pmatrix}\right\|_{L(\h)} < 2Ca_n.
\end{align*}
 The explicit rates obtained in \eqref{eq:explicit_rate} and \eqref{eq:explicit_rate2} prove that our approximation scheme satisfies (H2).

\subsection{The Algorithm}
It remains to extend the algorithm $\Theta_n^B$ from a single compact set $B\subset\mathbb C^\mathrm{ext}$ to the entire complex plane. This is done via a diagonal-type argument. 
\subsubsection{Odd Dimensions}
We choose a tiling of $\C$, where we start with a square $B_1=\big\{z\in\C\,\big|\,|\re(z)|\leq\f12, \, -1\leq|\im(z)|\leq 0 \big\}$ and then add squares in a counterclockwise spiral manner as shown in Figure \ref{fig:tiling}.

\begin{figure}[htbp]
	\centering
	\contourlength{2pt}
	\begin{tikzpicture}[>=stealth]
		\draw[-stealth] (-3,0)--(3,0) node[above]{$\mathrm{Re}\, z$};
		\draw[-stealth] (0,-2.5)--(0,2) node[right]{$\mathrm{Im}\, z$};
		\draw   (-0.05,-1)--(0.05,-1);
		\draw   (-0.05,-2)--(0.05,-2);
		\draw   (-0.05,-3)--(0.05,-3);
		\foreach \y in {0,-1,-2}
		\foreach \x in {0,1} {
			\draw (-0.5+\x,-2-\y) rectangle (0.5+\x,-1-\y);
		}
		\draw (-1.5,0) rectangle (-0.5,1);
		\node at (0,-0.5) {\contour{white}{$B_1$}};
		\node at (0,-1.5) {\contour{white}{$B_2$}};
		\node at (1,-1.5) {$B_3$};
		\node at (1,-0.5) {$B_4$};
		\node at (1,0.5) {$B_5$};
		\node at (0,0.5) {\contour{white}{$B_6$}};
		\node at (-1,0.5) {$B_7$};
		\contourlength{1pt}
		\node at (-1,-0.4) {\contour{white}{$\vdots$}};
	\end{tikzpicture}
	\caption{Tiling of the complex plane}
	\label{fig:tiling}
\end{figure}
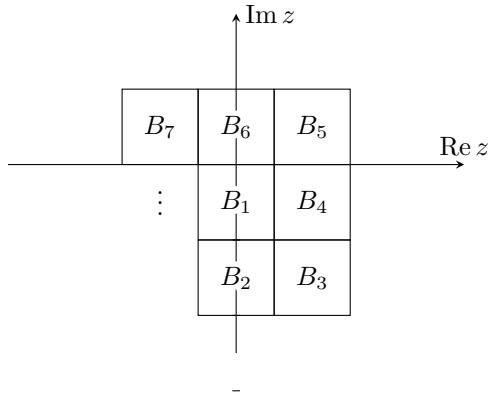
Next, we define our algorithm as follows. We let
\begin{align*}
	\Gamma_1(q) &:= \Theta_1^{B_1}(q) \\
	\Gamma_2(q) &:= \Theta_2^{B_1}(q) \cup \Theta_2^{B_2}(q) \\
	\Gamma_3(q) &:= \Theta_3^{B_1}(q) \cup \Theta_3^{B_2}(q) \cup \Theta_3^{B_3}(q) \\
	&\vdots\\
	\Gamma_n(q) &:= \bigcup_{j=1}^n \Theta_n^{B_j}(q).
\end{align*}
Lemma \ref{lemma:convergence} ensures that each $\Theta_n^{B_k}$ converges to $\Res(q)\cap B_k$ for fixed $k$ and since the $\{B_k\}$ form a tiling of $\C$, it follows that $\Gamma_n(q)\to \Res(q)$ in Attouch-Wets metric.

\subsubsection{Even Dimensions}

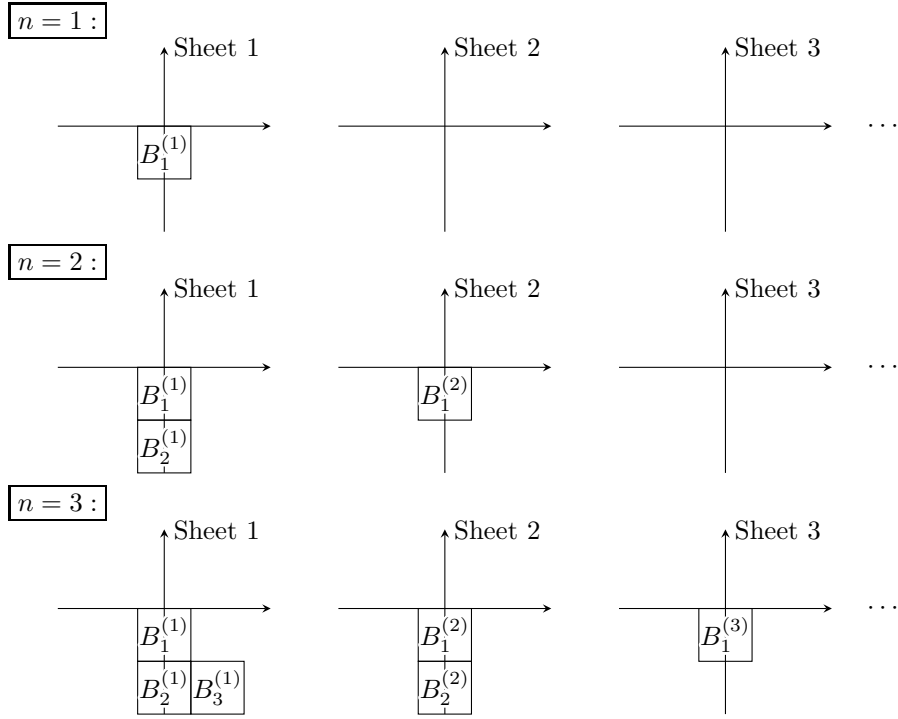
\begin{figure}[htbp]
	\centering
	\contourlength{1pt}
%
%

\begin{tikzpicture}[>=stealth,scale=0.7]
	\draw[-stealth] (-2,0)--(2,0);
	\draw[-stealth] (0,-2)--(0,1.5) node[right]{Sheet 1};
	\foreach \y in {-1}
	\foreach \x in {0} {
		\draw (-0.5+\x,-2-\y) rectangle (0.5+\x,-1-\y);
	}
	\node at (0,-0.5) {\contour{white}{$B_1^{(1)}$}};
	
\node at (-2,2) {$\boxed{n=1:}$};
	
	\begin{scope}[xshift=150]
		\draw[-stealth] (-2,0)--(2,0);
		\draw[-stealth] (0,-2)--(0,1.5) node[right]{Sheet 2};
	\end{scope}
	
	\begin{scope}[xshift=300]
		\draw[-stealth] (-2,0)--(2,0);
		\draw[-stealth] (0,-2)--(0,1.5) node[right]{Sheet 3};
		\draw (3,0) node{$\cdots$};
	\end{scope}
	
	
	\begin{scope}[xshift=0,yshift=-130]
\node at (-2,2) {$\boxed{n=2:}$};
		\draw[-stealth] (-2,0)--(2,0);
		\draw[-stealth] (0,-2)--(0,1.5) node[right]{Sheet 1};
		\foreach \y in {-1,0}
		\foreach \x in {0} {
			\draw (-0.5+\x,-2-\y) rectangle (0.5+\x,-1-\y);
		}
		\node at (0,-0.5) {\contour{white}{$B_1^{(1)}$}};
		\node at (0,-1.5) {\contour{white}{$B_2^{(1)}$}};
	\end{scope}
	
	\begin{scope}[xshift=150,yshift=-130]
		\draw[-stealth] (-2,0)--(2,0);
		\draw[-stealth] (0,-2)--(0,1.5) node[right]{Sheet 2};
		\foreach \y in {-1}
		\foreach \x in {0} {
			\draw (-0.5+\x,-2-\y) rectangle (0.5+\x,-1-\y);
		}
		\node at (0,-0.5) {\contour{white}{$B_1^{(2)}$}};
	\end{scope}
	
	\begin{scope}[xshift=300,yshift=-130]
		\draw[-stealth] (-2,0)--(2,0);
		\draw[-stealth] (0,-2)--(0,1.5) node[right]{Sheet 3};
		\draw (3,0) node{$\cdots$};
	\end{scope}
	
	\begin{scope}[xshift=0,yshift=-260]
\node at (-2,2) {$\boxed{n=3:}$};
		\draw[-stealth] (-2,0)--(2,0);
		\draw[-stealth] (0,-2)--(0,1.5) node[right]{Sheet 1};
		\foreach \y in {0,-1}
		\foreach \x in {0} {
			\draw (-0.5+\x,-2-\y) rectangle (0.5+\x,-1-\y);
		}
		\draw (0.5,-2) rectangle (1.5,-1);
		\node at (0,-0.5) {\contour{white}{$B_1^{(1)}$}};
		\node at (0,-1.5) {\contour{white}{$B_2^{(1)}$}};
		\node at (1,-1.5) {$B_3^{(1)}$};
	\end{scope}
	
	\begin{scope}[xshift=150,yshift=-260]
		\draw[-stealth] (-2,0)--(2,0);
		\draw[-stealth] (0,-2)--(0,1.5) node[right]{Sheet 2};
		\foreach \y in {-1,0}
		\foreach \x in {0} {
			\draw (-0.5+\x,-2-\y) rectangle (0.5+\x,-1-\y);
		}
		\node at (0,-0.5) {\contour{white}{$B_1^{(2)}$}};
		\node at (0,-1.5) {\contour{white}{$B_2^{(2)}$}};
	\end{scope}
	
	\begin{scope}[xshift=300,yshift=-260]
		\draw[-stealth] (-2,0)--(2,0);
		\draw[-stealth] (0,-2)--(0,1.5) node[right]{Sheet 3};
		\foreach \y in {-1}
		\foreach \x in {0} {
			\draw (-0.5+\x,-2-\y) rectangle (0.5+\x,-1-\y);
		}
		\node at (0,-0.5) {\contour{white}{$B_1^{(3)}$}};
		\draw (3,0) node{$\cdots$};
	\end{scope}
	
\end{tikzpicture}

	\caption{Tiling of the logarithmic Riemann surface}
	\label{fig:even-dim_tiling}
\end{figure}

In even dimensions we have to cover not only the complex plane $\C$, but its logarithmic covering space, which is equivalent to covering infinitely many copies of the complex plane. A similar strategy as in the odd dimensional case, together with a diagonal-type argument does the job in this case. Indeed, we can construct a cover by boxes $B_n$ as follows (cf. Figure \ref{fig:even-dim_tiling}). 
\begin{enumerate}
	\item Start with box $B_1$ (defined as in the odd dimensional case) on the first Riemann sheet;
	\item Add a box $B_2$ below $B_1$ on sheet number 1 and add a box $B_1$ on sheet number 2;
	\item Add a box $B_3$ on sheet number 1, add a box $B_2$ on sheet number 2 and a box $B_1$ on sheet number 3;
	\item \dots
\end{enumerate}

Next, define again
\begin{align*}
	\Gamma_1(q) &:= \Theta_1^{B_1^{(1)}}(q) \\
	\Gamma_2(q) &:= \Theta_2^{B_1^{(1)}}(q) \cup \Theta_2^{B_2^{(1)}}(q) \cup \Theta_2^{B_1^{(2)}}(q) \\
	\Gamma_3(q) &:= \Theta_3^{B_1^{(1)}}(q) \cup \Theta_3^{B_2^{(1)}}(q) \cup \Theta_3^{B_3^{(1)}}(q) \cup \Theta_3^{B_1^{(2)}}(q) \cup \Theta_3^{B_2^{(2)}}(q) \cup \Theta_3^{B_3^{(1)}}(q) \\
	&\vdots\\
	\Gamma_n(q) &:= \bigcup_{k=1}^n\bigcup_{j=1}^{n-k+1} \Theta_n^{B_j^{(k)}}(q).
\end{align*}
Lemma \ref{lemma:convergence} ensures that each $\Theta_n^{B_j^{(k)}}$ converges to $\Res(q)\cap B_j^{(k)}$ for fixed $k$ and since the $\{B_j^{(k)}\}$ form a tiling of $\C^{\mathrm{ext}}$, it follows that $\Gamma_n(q)\to \Res(q)$ in Attouch-Wets metric. The proof of Theorem \ref{th:mainth} is complete.

\section{Numerical Results}\label{sec:Numerics}
Software to compute resonances has been in existence for decades \cite{Rittby81, BEM00, AAD01}. The authors of \cite{BZ} recently proposed a collection of MATLAB codes to compute resonance poles and scattering of plane waves efficiently (``MatScat'', cf. \cite{Bindel}). In this section we compare the results of our  algorithm to that of MatScat.

In order to study the actual numerical performance of our algorithm, we coded a MATLAB routine for the one-dimensional case with $\supp(q)\subset[a,b]$ (for some known $a<b$), which computes the set 
$$\Big\{z\in \mathcal{L}_n\cap B\,\Big|\,\Big\|\Big(\mathds{1}_{n\times n}+\big(K(i,j)\big)_{i,j\in \f{b-a}{n}\Z\cap[a,b]}\Big)^{-1}\Big\|>C\Big\},$$ 
where the region $B$ in the complex plane, the lattice distance of $\mathcal{L}_n$ and cutoff threshold $C$ were treated as independent parameters.

\paragraph{Comparison of results.} 
Figures \ref{fig:Gauss_well} and \ref{fig:Trapping} show the output of MatScat (black dots) versus the output of our algorithm (blue regions) for a Gaussian well and trapping potential, respectively. As the plots show, there is agreement between the two.

\begin{figure}[htbp]
\centering
	\input{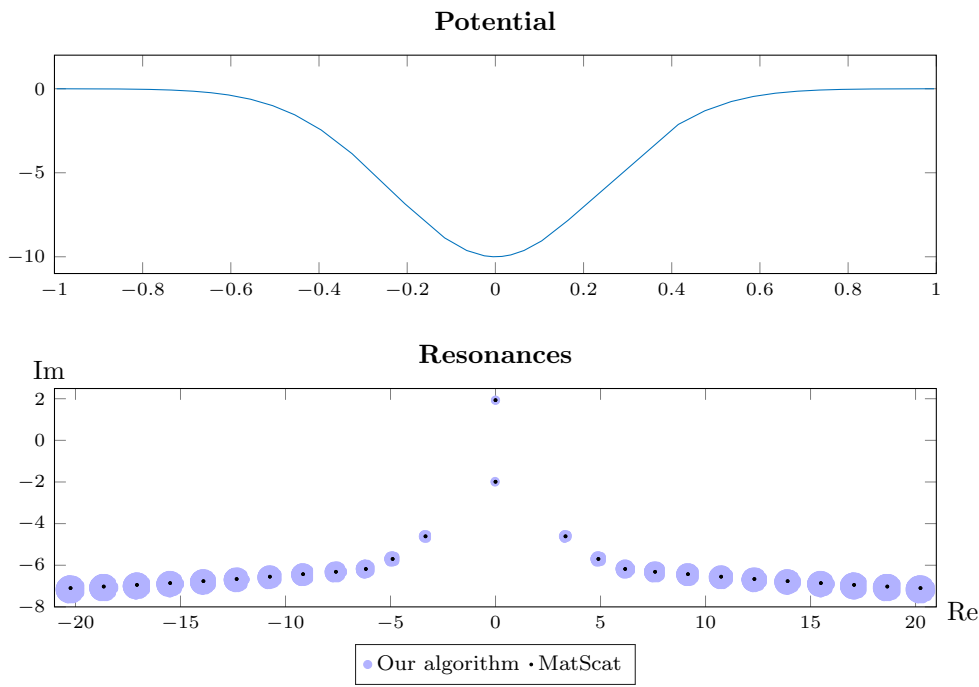}
	\caption{Comparison of the result of \cite{Bindel} (black) and our algorithm (blue) for a Gaussian well supported between $-1$ and $1$. The chosen parameter values are: $n=100$; threshold for resolvent norm: {$C=200$}; number of lattice points in the shown region of the complex plane: $M\times 4M = 1000\times 4000$.}
	\label{fig:Gauss_well}
\end{figure}

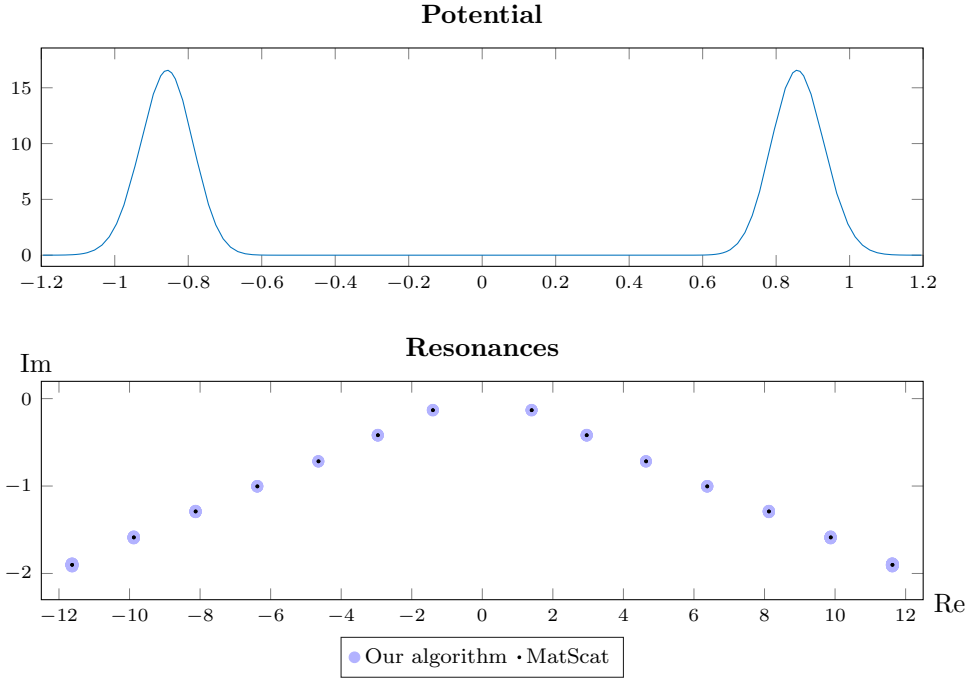
\begin{figure}[htbp]
\centering
%
%

%

\definecolor{mycolor1}{rgb}{0.00000,0.44700,0.74100}%
\definecolor{mycolor2}{rgb}{0.85000,0.32500,0.09800}%
\definecolor{mycolor3}{rgb}{0.92900,0.69400,0.12500}%
\pgfplotsset{every tick label/.append style={font=\scriptsize}}
\begin{tikzpicture}

\begin{axis}[%
width=0.761\textwidth,
height=0.19\textwidth,
at={(0\textwidth,0.29\textwidth)},
scale only axis,
xmin=-1.2,
xmax=1.2,
ymin=-1,
ymax=18.5650934483496,
axis background/.style={fill=white},
title style={font=\bfseries},
title={Potential},
legend style={legend cell align=left, align=left, draw=white!15!black},
legend style={at={(0.5,-0.17)},anchor=north},
legend columns=2
]
\addplot [color=mycolor1]
  table[row sep=crcr]{%
-1.195	0.000717238130349074\\
-1.155	0.0061863920263292\\
-1.135	0.0165912947722511\\
-1.115	0.0418396490982396\\
-1.105	0.0649093315649587\\
-1.095	0.0991309233588957\\
-1.085	0.149020866945307\\
-1.075	0.220481747356338\\
-1.055	0.459924570156364\\
-1.035	0.898952325922703\\
-1.015	1.64469030139885\\
-0.995000000000001	2.81361178012681\\
-0.975000000000001	4.49555812271544\\
-0.945	7.96698950550366\\
-0.895	14.4298594035095\\
-0.875	16.0661043264973\\
-0.864999999999998	16.4728968090474\\
-0.855	16.5650934483496\\
-0.844999999999999	16.3336775436047\\
-0.835000000000001	15.7884096259715\\
-0.815000000000001	13.8839856848365\\
-0.774999999999999	8.36227088392535\\
-0.745000000000001	4.55240487336904\\
-0.725000000000001	2.71000658055781\\
-0.704999999999998	1.46848861098815\\
-0.684999999999999	0.722056613208235\\
-0.664999999999999	0.321055536736633\\
-0.645	0.128608126500584\\
-0.625	0.0462224667698656\\
-0.605	0.0148380699322175\\
-0.585000000000001	0.00423337773480625\\
-0.555	0.000510916750116053\\
-0.405000000000001	8.32542923490109e-11\\
0.574999999999999	0.00215974493698212\\
0.594999999999999	0.00804594623955879\\
0.605	0.0148380699322175\\
0.614999999999998	0.0265691421982588\\
0.625	0.0462224667698656\\
0.635000000000002	0.0781747818540772\\
0.645	0.128608126500584\\
0.655000000000001	0.205919261669944\\
0.664999999999999	0.321055536736633\\
0.675000000000001	0.487680722912131\\
0.695	1.04252410243703\\
0.715	2.0188640213759\\
0.734999999999999	3.55326840335787\\
0.754999999999999	5.70116112098864\\
0.795000000000002	11.241633563249\\
0.824999999999999	14.9572439809035\\
0.844999999999999	16.3336775436047\\
0.855	16.5650934483496\\
0.864999999999998	16.4728968090474\\
0.875	16.0661043264972\\
0.895	14.4298594035095\\
0.925000000000001	10.6732543039576\\
0.965	5.53679165861591\\
0.995000000000001	2.8136117801268\\
1.015	1.64469030139885\\
1.035	0.898952325922728\\
1.055	0.459924570156364\\
1.075	0.220481747356342\\
1.095	0.0991309233588957\\
1.115	0.0418396490982396\\
1.135	0.0165912947722511\\
1.155	0.0061863920263292\\
1.185	0.00125709945772456\\
1.195	0.000717238130349074\\
};


\end{axis}

\begin{axis}[%
width=0.761\textwidth,
height=0.19\textwidth,
at={(0\textwidth,0\textwidth)},
scale only axis,
xmin=-12.5,
xmax=12.4975,
ymin=-2.3,
ymax=0.1975,
axis background/.style={fill=white},
title style={font=\bfseries},
title={Resonances},
legend style={legend cell align=left, align=left, draw=white!15!black},
legend style={at={(0.5,-0.17)},anchor=north},
legend columns=2,
x label style={at={(axis description cs:1.03,0.25)},anchor=north},
xlabel={Re},
y label style={at={(axis description cs:0.1,1)}, rotate=-90, anchor=south}, 
ylabel={Im}
]
\addplot [color=mycolor1, draw=none, mark size=2pt, mark=*, mark options={solid, white!70!blue}, only marks]
  table[row sep=crcr]{%
-11.645	-1.895\\
-11.645	-1.8975\\
-11.645	-1.9\\
-11.645	-1.9025\\
-11.645	-1.905\\
-11.645	-1.9075\\
-11.6425	-1.89\\
-11.6425	-1.8925\\
-11.6425	-1.91\\
-11.6425	-1.9125\\
-11.64	-1.8875\\
-11.64	-1.89\\
-11.64	-1.8925\\
-11.64	-1.895\\
-11.64	-1.8975\\
-11.64	-1.9\\
-11.64	-1.9025\\
-11.64	-1.905\\
-11.64	-1.9075\\
-11.64	-1.91\\
-11.64	-1.9125\\
-11.64	-1.915\\
-11.6375	-1.9175\\
-11.635	-1.885\\
-11.6325	-1.92\\
-11.62	-1.885\\
-11.62	-1.8875\\
-11.62	-1.89\\
-11.62	-1.8925\\
-11.62	-1.895\\
-11.62	-1.8975\\
-11.62	-1.9\\
-11.62	-1.9025\\
-11.62	-1.905\\
-11.62	-1.9075\\
-11.62	-1.91\\
-11.62	-1.9125\\
-11.62	-1.915\\
-11.62	-1.9175\\
-11.62	-1.92\\
-9.89	-1.5825\\
-9.89	-1.585\\
-9.89	-1.5875\\
-9.89	-1.59\\
-9.89	-1.5925\\
-9.8875	-1.58\\
-9.8875	-1.595\\
-9.885	-1.5975\\
-9.8825	-1.575\\
-9.8825	-1.6\\
-9.8725	-1.5775\\
-9.8725	-1.5825\\
-9.8725	-1.585\\
-9.8725	-1.5875\\
-9.8725	-1.59\\
-9.8725	-1.5925\\
-9.8725	-1.595\\
-9.8725	-1.5975\\
-8.135	-1.2875\\
-8.135	-1.29\\
-8.135	-1.2925\\
-8.1325	-1.285\\
-8.1325	-1.295\\
-8.1325	-1.2975\\
-8.13	-1.2825\\
-8.1275	-1.3\\
-8.125	-1.2825\\
-8.125	-1.285\\
-8.125	-1.2875\\
-8.125	-1.29\\
-8.125	-1.2925\\
-8.125	-1.295\\
-8.125	-1.2975\\
-8.125	-1.3\\
-6.3825	-1\\
-6.3825	-1.0025\\
-6.3825	-1.005\\
-6.3825	-1.0075\\
-6.38	-0.9975\\
-6.3775	-0.9975\\
-6.3775	-1\\
-6.3775	-1.0025\\
-6.3775	-1.005\\
-6.3775	-1.0075\\
-4.65	-0.717499999999999\\
-4.6475	-0.7125\\
-4.6475	-0.715\\
-4.6475	-0.720000000000001\\
-4.6475	-0.7225\\
-2.965	-0.414999999999999\\
-2.965	-0.4175\\
-2.965	-0.42\\
-2.965	-0.422499999999999\\
-2.9625	-0.414999999999999\\
-2.9625	-0.4175\\
-2.9625	-0.42\\
-2.9625	-0.422499999999999\\
-2.96	-0.4125\\
-2.96	-0.425000000000001\\
-1.405	-0.130000000000001\\
-1.405	-0.1325\\
-1.4025	-0.1275\\
-1.4025	-0.137499999999999\\
-1.395	-0.130000000000001\\
-1.395	-0.1325\\
1.395	-0.130000000000001\\
1.395	-0.1325\\
1.3975	-0.1275\\
1.3975	-0.130000000000001\\
1.3975	-0.1325\\
1.3975	-0.137499999999999\\
2.955	-0.414999999999999\\
2.955	-0.4175\\
2.955	-0.42\\
2.955	-0.422499999999999\\
2.9575	-0.4125\\
2.9575	-0.425000000000001\\
2.965	-0.414999999999999\\
2.965	-0.4175\\
2.965	-0.42\\
2.965	-0.422499999999999\\
4.6375	-0.715\\
4.6375	-0.717499999999999\\
4.6375	-0.720000000000001\\
4.64	-0.7125\\
4.64	-0.7225\\
6.37	-1\\
6.37	-1.0025\\
6.37	-1.005\\
6.37	-1.0075\\
6.3725	-0.9975\\
6.38	-0.9975\\
6.38	-1\\
6.38	-1.0025\\
6.38	-1.005\\
6.38	-1.0075\\
8.1175	-1.285\\
8.1175	-1.2875\\
8.1175	-1.29\\
8.1175	-1.2925\\
8.1175	-1.295\\
8.12	-1.2825\\
8.12	-1.2975\\
8.1225	-1.3\\
8.1275	-1.2825\\
8.1275	-1.285\\
8.1275	-1.2875\\
8.1275	-1.29\\
8.1275	-1.2925\\
8.1275	-1.295\\
8.1275	-1.2975\\
8.1275	-1.3\\
9.865	-1.585\\
9.865	-1.5875\\
9.865	-1.59\\
9.8675	-1.58\\
9.8675	-1.5825\\
9.8675	-1.5925\\
9.8675	-1.595\\
9.87	-1.5975\\
9.875	-1.575\\
9.875	-1.5775\\
9.875	-1.5825\\
9.875	-1.585\\
9.875	-1.5875\\
9.875	-1.59\\
9.875	-1.5925\\
9.875	-1.595\\
9.875	-1.5975\\
9.875	-1.6\\
11.6075	-1.8975\\
11.6075	-1.9\\
11.6075	-1.9025\\
11.6075	-1.905\\
11.61	-1.8925\\
11.61	-1.895\\
11.61	-1.9075\\
11.61	-1.91\\
11.6125	-1.89\\
11.6125	-1.9125\\
11.6125	-1.915\\
11.615	-1.8875\\
11.6175	-1.885\\
11.6175	-1.9175\\
11.62	-1.92\\
11.6225	-1.8825\\
11.6225	-1.8875\\
11.6225	-1.89\\
11.6225	-1.8925\\
11.6225	-1.895\\
11.6225	-1.8975\\
11.6225	-1.9\\
11.6225	-1.9025\\
11.6225	-1.905\\
11.6225	-1.9075\\
11.6225	-1.91\\
11.6225	-1.9125\\
11.6225	-1.915\\
11.6225	-1.9175\\
11.6225	-1.92\\
11.6425	-1.89\\
11.6425	-1.8925\\
11.6425	-1.895\\
11.6425	-1.8975\\
11.6425	-1.9\\
11.6425	-1.9025\\
11.6425	-1.905\\
11.6425	-1.9075\\
11.6425	-1.91\\
11.6425	-1.9125\\
11.645	-1.9075\\
};
\addlegendentry{\footnotesize Our algorithm\:}

\addplot [color=mycolor2, draw=none, mark size=0.5pt, mark=*, mark options={solid, black}, only marks]
  table[row sep=crcr]{%
-1.40028634825695	-0.132603009013669\\
1.40028634825695	-0.132603009013669\\
-2.95958730313506	-0.418766595829142\\
2.95958730313506	-0.418766595829142\\
-4.64320818119665	-0.717265044440312\\
4.64320818119665	-0.717265044440312\\
-6.37651806261088	-1.00390536971453\\
6.37651806261088	-1.00390536971453\\
-8.12589982881437	-1.29047174104079\\
8.12589982881437	-1.29047174104079\\
-9.87801083111123	-1.58738047888135\\
9.87801083111123	-1.58738047888135\\
-11.6265951763612	-1.90102583851946\\
11.6265951763612	-1.90102583851946\\
};
\addlegendentry{\footnotesize MatScat}

\addplot [color=mycolor3, draw=none, mark size=0.5pt, mark=*, mark options={solid, black}, forget plot]
  table[row sep=crcr]{%
-1.40028634825421	-0.132603009014058\\
1.40028634825421	-0.132603009014058\\
-2.9595873031987	-0.418766595745865\\
2.9595873031987	-0.418766595745865\\
-4.64320818116714	-0.717265044468375\\
4.64320818116714	-0.717265044468375\\
-6.37651806282499	-1.00390536979603\\
6.37651806282499	-1.00390536979603\\
-8.12589982866485	-1.29047174110078\\
8.12589982866485	-1.29047174110078\\
-9.87801083146072	-1.58738047857225\\
9.87801083146072	-1.58738047857225\\
-11.6265951764263	-1.90102583968939\\
11.6265951764263	-1.90102583968939\\
};
\end{axis}
\end{tikzpicture}%
	\caption{Comparison of the result of \cite{Bindel} (black) and our algorithm (blue) for a smooth trapping potential supported between $-1.2$ and $1.2$. The chosen parameter values are: $n=100$; threshold for resolvent norm: {$C=200$}; number of lattice points in the shown region of the complex plane: $M\times 10M = 1000\times 10000$.}
	\label{fig:Trapping}
\end{figure}

\paragraph{Limitations.}
As mentioned before, MatScat has been developed with the goal to create an efficient algorithm to compute resonances fast. Indeed, the computation of the black dots in Figure \ref{fig:Gauss_well} takes less than a second, while computing the regions with our algorithm takes several hours on a personal computer. We stress that our MATLAB code was written mainly for illustration purposes and that there is considerable room for improvement in numerical efficiency. 
Moreover, our algorithm can only yield reliable results in a certain region, as the following heuristic calculations make clear.
\begin{itemize}
\item\emph{Imaginary part of $z$:} Since the fundamental solution $G(x,z)=\f{1}{2iz}e^{\i z|x|}$ grows exponentially with $-\im(z)$ and $x\in[-a,a]$, a limit is reached when $|\im(z)|\sim \f{\log(2M)}{2a}$, where $M$ is the largest number the machine can store with adequate precision (for the interval $[-a,a]=[-1,1]$ and $M=10^{16}$ this bound yields $\im(z)\gtrsim -18.8$).

\item\emph{Real part of $z$:} Similarly, a natural bound on $\re(z)$ is reached when the period of $e^{\i z|x|}$ is less than twice the lattice spacing $\f{2}{n}$, i.e. when
$|\re(z)|\lesssim \pi n$ (for $n=30$ this bound yields $|\re(z)|\lesssim 94$).
\end{itemize}
Numerical experiments have confirmed the above bounds (see Figure \ref{fig:Bound_on_real_part}). Note that the bound on $\im(z)$ is fixed by the machine precision, while the bound on $|\re(z)|$ can be raised by increasing $n$.

\begin{remark}
We note that our algorithm is not restricted to one-dimension or real-valued potentials. Indeed,  the algorithm $\Gamma_n$ only uses the bound $\supp(q)\subset Q_M$, and higher dimensional implementations of $\Gamma_n$ can be coded similarly to the one-dimensional one.
\end{remark}
\begin{figure}[htbp]
	\centering
	\input{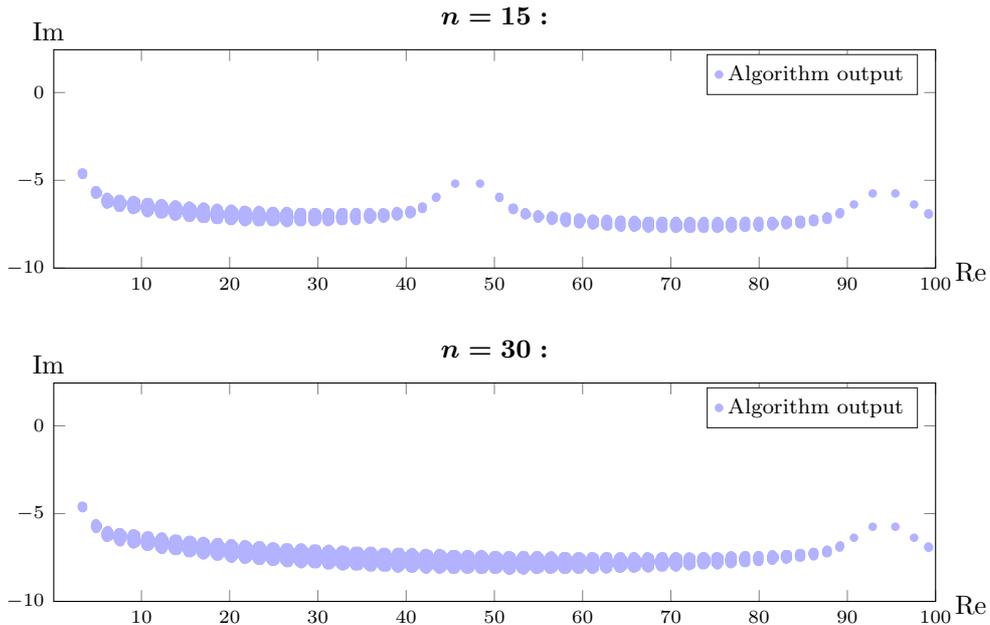}
	\caption{Numerical artefacts for large real part of $z$. Top: Output of our algorithm for Gaussian well potential on the interval $[-1,1]$ with $n=15$. Bottom: Output for the same problem with $n=30$. The locations of the spurious peaks agree with the bound $|\re(z)|\sim \pi n$ in each case.}
	\label{fig:Bound_on_real_part}
\end{figure}

\appendix

\section{Fundamental Solution}\label{app:greens}
In this appendix we  gather some well-known results about the fundamental solution for the Helmholtz equation. These facts are used to show that the abstract framework of Section \ref{sec:abstract} holds in the context of our algorithm as defined in Section \ref{sec:algorithm}, namely that eqs. \eqref{eq:K-KnPn}-\eqref{eq:K-PKP} hold.
We adopt the notation of \cite{AS} and write $f(\zeta)\sim \zeta^\nu$ if $f$ and $\zeta^\nu$ are \emph{asymptotically equal}, i.e. $|f(\zeta)-\zeta^\nu|=O(|\zeta|^{\nu+1})$ as $|\zeta|\to 0$.
\begin{remark}
	By the asymptotic expansion of the Hankel functions
	\begin{equation*}
		H^{(1)}_\nu(\zeta)\sim \begin{cases}
		-\f{\Gamma(\nu)}{\pi}\left(\f{\zeta}{2}\right)^{-\nu}, & \nu>0,\\
 			\f{2\i}{\pi}\log(\zeta), & \nu=0,
	 \end{cases}
	\end{equation*}
	where $\Gamma$ denotes the Gamma function and $\log$ denotes the principal branch of the logarithm (cf. \cite[Ch. 9.1.9]{AS}), we find that the fundamental solution \eqref{eq:Fundamental_Solution} satisfies the small $|x|$ asymptotics
	\begin{align*}
		G(x,z) &\sim -\f{\i\Gamma(\f{d-2}{2})}{\pi}\left(\f{z|x|}{2}\right)^{-\f{d-2}{2}}\f\i 4\left( \f{z}{2\pi |x|} \right)^{\f{d-2}{2}}  \\
		&= \f{\Gamma(\f{d-2}{2})}{4\pi^{\f n 2}}\f{1}{|x|^{d-2}} , \qquad  \text{as }|x|\to 0,
	\end{align*}
	for $n\geq 3$, and
	\begin{align*}
		G(x,z) &\sim -\f{1}{2\pi}\log(z|x|), \qquad \text{as }|x|\to 0,
	\end{align*}
	for $n=2$. Hence
	\begin{align}\label{eq:G_bound}
		|G(x,z)| &\leq C_z\cdot \begin{cases}
		\f{1}{|x|^{d-2}}, &n\geq 3,\\
 		 \log(|x|), &n=2,
	 \end{cases}
	\end{align}
	where $C_z>0$ is uniformly bounded for $z$ in a compact subset of $\C$. Similar formulas hold for the derivatives of $G$. Indeed, identities for Hankel functions (cf. \cite[Ch. 9.1.30]{AS}) show that
	\begin{align}\label{eq:G_bound_2}
		|\nabla G(x,z)| &\leq 
		\f{C_z}{|x|^{d-1}},\qquad\text{ for }d\geq 2.
	\end{align}
\end{remark}

\begin{remark}\label{remark:even_odd}
	From the representation of $G(x,z)$ in terms of Hankel functions it follows that $G$ can be  continued analytically in $z$ through the branch cut $\R_+$. In fact, it can be shown that $G$ can be continued to
	\begin{itemize}
		\item the Riemann surface of the complex square root, if $d$ is odd,
		\item the Riemann surface of the complex logarithm, if $d$ is even,
	\end{itemize}
	(cf. \cite[Ch. 3.1.4]{DZ}). The estimates \eqref{eq:G_bound} and \eqref{eq:G_bound_2} remain valid in either case.
\end{remark}

\end{document}